\documentclass[a4paper,english,fontsize=10pt,parskip=half,abstracton]{scrartcl}
\usepackage{babel}
\usepackage[utf8]{inputenc}
\usepackage[T1]{fontenc}
\usepackage[a4paper,left=20mm,right=20mm,top=30mm,bottom=30mm]{geometry}
\usepackage{amsmath}
\usepackage{amsthm}
\usepackage{amssymb}
\usepackage{enumerate}
\usepackage[bookmarks=false,
            pdftitle={Cartan matrices and Brauer's k(B)-conjecture},
            pdfauthor={Benjamin Sambale},
            pdfkeywords={Cartan matrices, Brauer's k(B)-conjecture, block theory, decomposition matrices, quadratic forms},
            pdfstartview={FitH}]{hyperref}
\newtheorem{Lemma}{Lemma}
\newtheorem{Theorem}{Theorem}
\newtheorem{Proposition}{Proposition}
\newtheorem{Corollary}{Corollary}
\newtheorem*{Question}{Question A}

\addto\extrasenglish{

}

\renewcommand{\phi}{\varphi}
\newcommand{\C}{\operatorname{C}}
\newcommand{\N}{\operatorname{N}}
\newcommand{\Z}{\operatorname{Z}}
\newcommand{\A}{\operatorname{A}}
\newcommand{\Aut}{\operatorname{Aut}}

\newcommand{\GL}{\operatorname{GL}}
\newcommand{\SL}{\operatorname{SL}}
\newcommand{\Irr}{\operatorname{Irr}}
\newcommand{\IBr}{\operatorname{IBr}}

\mathchardef\ordinarycolon\mathcode`\:  
 \mathcode`\:=\string"8000
 \begingroup \catcode`\:=\active
   \gdef:{\mathrel{\mathop\ordinarycolon}}
 \endgroup

\title{Cartan matrices and Brauer's $k(B)$-conjecture}
\author{
Benjamin Sambale\\
Mathematisches Institut\\
Friedrich-Schiller-Universität\\
07743 Jena\\
Germany\\
{\tt benjamin.sambale@uni-jena.de}
}
\date{\today}

\begin{document}

\frenchspacing
\maketitle

\begin{abstract} 
It is well known that the Cartan matrix of a block of a finite group cannot be arranged as a direct sum of smaller matrices. In this paper we address the question if this remains true for equivalent matrices. The motivation for this question comes from the work by Külshammer and Wada \cite{KuelshammerWada}, which contains certain bounds for the number of ordinary characters in terms of Cartan invariants. As an application we prove such a bound in the special case, where the determinant of the Cartan matrix coincides with the order of the defect group. 

In the second part of the paper we show that Brauer's $k(B)$-conjecture holds for $2$-blocks under some restrictions on the defect group. For example, the $k(B)$-conjecture holds for $2$-blocks if the corresponding defect group is a central extension of a metacyclic group by a cyclic group. The same is true if the defect group contains a central cyclic subgroup of index at most $9$. In particular the $k(B)$-conjecture holds for $2$-blocks of defect at most $4$ and $3$-blocks of defect at most $3$. Using the classification of finite simple groups we improve this result to central cyclic subgroups of index $16$ with one possible exception.
In particular the $k(B)$-conjecture holds for $2$-blocks of defect $5$, except possible the extraspecial defect group $D_8\ast D_8$. 
As a byproduct, we obtain the block invariants for $2$-blocks with minimal nonmetacyclic defect groups. Some proofs rely on computer computations with GAP \cite{GAP4}. The paper is a part of the author's PhD thesis.
\end{abstract}


\tableofcontents

\section{Introduction}
Let $G$ be a finite group and let $B$ be a $p$-block of $G$ for a prime number $p$. We denote the inertial index of $B$ by $e(B)$, the number of ordinary irreducible characters by $k(B)$, and the number of irreducible Brauer characters by $l(B)$. Moreover, let $d$ be the defect of $B$.

It is well known that the Cartan matrix $C$ of $B$ is indecomposable as integer matrix, i.\,e. there is no arrangement of the indecomposable projective modules such that $C$ splits into a direct sum of smaller matrices (recall that $C$ is symmetric). 

We call two matrices $A,B\in\mathbb{Z}^{l\times l}$ \emph{equivalent} if there exists a matrix $S\in\operatorname{GL}(l,\mathbb{Z})$ with $A=S^{\text{T}}BS$, where $S^{\text{T}}$ denotes the transpose of $S$. Every symmetric matrix gives rise to a quadratic form. In this sense equivalent matrices describe equivalent quadratic forms. Richard Brauer describes equivalence of Cartan matrices via so called “basic sets”. 
He also studied Cartan matrices by applying the theory of quadratic forms (see \cite{Brauercertain}).
In general the property “being indecomposable” is not shared among equivalent matrices. For example $A=\bigl(\begin{smallmatrix}1&1\\1&2\end{smallmatrix}\bigr)$ is indecomposable, but $\bigl(\begin{smallmatrix}1&-1\\0&1\end{smallmatrix}\bigr)^{\text{T}}A\bigl(\begin{smallmatrix}1&-1\\0&1\end{smallmatrix}\bigr)=\bigl(\begin{smallmatrix}1&0\\0&1\end{smallmatrix}\bigr)$ is not. However, we were not able to find a Cartan matrix of a block which provides an equivalent decomposable matrix. So we raise the question:

\begin{Question}
Do there exist a Cartan matrix $C$ of a block $B$ and a matrix $S\in\operatorname{GL}(l(B),\mathbb{Z})$ such that $S^{\rm{T}}CS$ is decomposable?
\end{Question}

The motivation for this question comes from the fact that $k(B)$ can be bounded in terms of Cartan invariants:

\begin{Proposition}[see Theorem~A in \cite{KuelshammerWada}]\label{kulswada}
Let $B$ be a block with Cartan matrix $C=(c_{ij})$ up to equivalence. Then for every positive definite, integral quadratic form $q:=\sum_{1\le i\le j\le l(B)}{q_{ij}X_iX_j}$ we have
\[k(B)\le\sum_{1\le i\le j\le l(B)}{q_{ij}c_{ij}}.\]
In particular 
\begin{equation}\label{KW}
k(B)\le\sum_{i=1}^{l(B)}{c_{ii}}-\sum_{i=1}^{l(B)-1}{c_{i,i+1}}.
\end{equation}
\end{Proposition}

The point is that the inequalities are significantly sharper for indecomposable matrices. We illustrate this fact with an example. Let $l(B)=p=2$ and assume that the elementary divisors of $C$ are $2$ and $16$ (this happens for the principal block of $\GL(2,3)$). Then $C$ has the form
\[\begin{pmatrix}2&0\\0&16\end{pmatrix}\text{ or }\begin{pmatrix}6&2\\2&6\end{pmatrix}\]
up to equivalence. Inequality~\eqref{KW} gives $k(B)\le 18$ in the first case and $k(B)\le 10$ in the second.

\section{Upper bounds for $k(B)$}

We give an affirmative answer to question~A in two special cases.

\begin{Lemma}
Let $G$ be $p$-solvable and $l:=l(B)\ge 2$. Then there is no matrix $S\in\operatorname{GL}(l,\mathbb{Z})$ such that $S^{\rm{T}}CS=\bigl(\begin{smallmatrix}p^d&0\\0&C_1\end{smallmatrix}\bigr)$ with $C_1\in\mathbb{Z}^{(l-1)\times(l-1)}$. In particular $C$ is not equivalent to a diagonal matrix.
\end{Lemma}
\begin{proof}
Assume the contrary, i.\,e. there is a matrix $S=(s_{ij})\in\operatorname{GL}(l,\mathbb{Z})$ such that
\[C=(c_{ij})=S^{\text{T}}\begin{pmatrix}p^d&0\\0&C_1\end{pmatrix}S\]
with $C_1\in\mathbb{Z}^{(l-1)\times(l-1)}$. Let $s_i:=(s_{2i},s_{3i},\ldots,s_{li})$ for $i=1,\ldots,l$. By Theorem~(3H) in \cite{Fongpsolv} we have
\[p^ds_{i1}^2+s_iC_1s_i^{\text{T}}=c_{ii}\le p^d\]
for $i=1,\ldots,l$. Since $S$ is invertible, there exists $i$ such that $s_{1i}\ne 0$. We may assume $s_{11}\ne 0$. Then $s_{11}=\pm1$ and $s_1=(0,\ldots,0)$, because $C_1$ is positive definite. Now all other columns of $S$ are linearly independent of the first column. This gives $s_{1i}=0$ for $i=2,\ldots,l$. Hence, $S$ has the form $S=\bigl(\begin{smallmatrix}\pm1&0\\0&S_1\end{smallmatrix}\bigr)$ with $S_1\in\operatorname{GL}(l-1,\mathbb{Z})$. But then $C$ also has the form $\bigl(\begin{smallmatrix}p^d&0\\0&C_2\end{smallmatrix}\bigr)$ with $C_2\in\mathbb{Z}^{(l-1)\times(l-1)}$, a contradiction. The second claim follows at once, since $p^d$ is always an elementary divisor of $C$.
\end{proof}

Unfortunately the bound for the Cartan invariants used in the proof does not hold for arbitrary groups (see \cite{Landrock}).

\begin{Lemma}\label{decomposable}
If $\det C=p^d$, then for every $S\in\operatorname{GL}(l(B),\mathbb{Z})$ the matrix $S^{\rm{T}}CS$ is indecomposable.
\end{Lemma}
\begin{proof}
Again assume the contrary, i.\,e. there is a matrix $S\in\operatorname{GL}(l(B),\mathbb{Z})$ such that
\[C=S^{\text{T}}\begin{pmatrix}C_1&0\\0&C_2\end{pmatrix}S\]
with $C_1\in\mathbb{Z}^{m\times m}$ and $C_2\in\mathbb{Z}^{(l-m)\times(l-m)}$, where $l:=l(B)$ and $1\le m<l$. In particular $l<k(B)=:k$, because $l\ge 2$.
Since $\det C=p^d$, the elementary divisors of $C$ are $1$ and $p^d$, where $p^d$ occurs with multiplicity one. W.\,l.\,o.\,g. we may assume $\det C_1=1$. Let $Q=(q_{ij})$ be the corresponding part of the decomposition matrix, i.\,e. $Q^{\text{T}}Q=C_1$. By the Binet-Cauchy formula we have
\[1=\det C_1=\sum_{\substack{V\subseteq\{1,\ldots,k\},\\|V|=m}}{\det Q_V^{\text{T}}Q_V},\]
where $Q_V$ is the $m\times m$ submatrix consisting of the entries $\{q_{ij}:i\in V,\ j\in\{1,\ldots,m\}\}$. Since $\det Q_V^{\text{T}}Q_V\ge 0$, one summand is $1$ while the others are all $0$. Thus we may assume, that the first $m$ rows $q_1,\ldots,q_m$ of $Q$ are linearly independent. Now consider a row $q_i$ for $i\in\{m+1,\ldots,k\}$. Then $q_i$ is a rational linear combination of $q_2,\ldots,q_m$, because $q_2,\ldots,q_m,q_i$ are linearly dependent. By the same argument, $q_i$ is also a linear combination of $q_1,\ldots,q_{j-1},q_{j+1},\ldots,q_m$ for $j=2,\ldots,m$. This forces $q_i=(0,\ldots,0)$. Hence, all the rows $q_{m+1},\ldots,q_k$ vanish. Now consider a column $d(u)$ of generalized decomposition numbers, where $u$ is a nontrivial element of a defect group of $B$. By the orthogonality relations the scalar product of $d(u)$ and an arbitrary column of $Q$ vanishes. This means the first $m$ entries of $d(u)$ must be zero. Since this holds for all columns $d(u)$ with $u\ne 1$, there exists an irreducible character of $B$ which vanishes on the $p$-singular elements of $G$. It is well known that this is equivalent to $d=0$. But this contradicts $l\ge 2$.
\end{proof} 

As an application, we prove an upper bound for $k(B)$ in the case $\det C=p^d$. In the proof we will use the reduction theory of quadratic forms.

\begin{Theorem}\label{detCD}
If $l(B)\le 4$ and $\det C=p^d$, then 
\[k(B)\le\frac{p^d-1}{l(B)}+l(B).\]
Moreover, this bound is sharp.
\end{Theorem}
\begin{proof}
For $l:=l(B)=1$ the assertion is clear (see e.\,g. Corollary~5 in \cite{Olssoninequ}). So let $l\ge2$. Let $A=(a_{ij})$ be a reduced matrix in the sense of Minkowski which is equivalent to $C$ (see e.\,g. \cite{Waerden}). In particular we have $1\le a_{11}\le a_{22}\le\ldots\le a_{ll}$ and $2|a_{ij}|\le\min\{a_{ii},a_{jj}\}$ for $i\ne j$. For convenience we write $\alpha:=a_{11}$, $\beta:=a_{22}$ and so on.

We are going to apply inequality~\eqref{KW}. In order to do so, we will bound the trace of $A$ from above and the sum $a_{12}+a_{23}+\ldots+a_{l-1,l}$ from below.

Let $l(B)=2$. By \autoref{decomposable} we have $a_{12}\ne 0$ and $a_{12}>0$ after a suitable change of signs (i\,e. replacing $A$ by an equivalent matrix). By \cite{Barnes} we have $4\alpha\beta-\alpha^2\le4p^d$, so that
\begin{equation}\label{feasible}
\alpha+\beta\le\frac{5}{4}\alpha+\frac{p^d}{\alpha}=:f(\alpha).
\end{equation}
Since $2|a_{ij}|\le\min\{a_{ii},a_{jj}\}$, we have $2\le\alpha$, and $\alpha\le\beta$ yields $\alpha\le2\sqrt{p^d/3}$.
The convex function $f(\alpha)$ takes its maximal value in the interval $[2,2\sqrt{p^d/3}]$ on one of the two borders. 
An easy calculation shows $(p^d+5)/2=f(2)>f(2\sqrt{p^d/3})$ for $p^d\ge 9$. In case $p^d\le 6$ only $\alpha=2$ is possible. In the remaining cases we have $\alpha+\beta\le f(2)$ for all feasible pairs $(\alpha,\beta)$ (we call a pair $(\alpha,\beta)$ feasible if it satisfies inequality~\eqref{feasible}). 
Inequality~\eqref{KW} yields
\[k(B)\le\alpha+\beta-a_{12}\le f(2)-1=\frac{p^d-1}{l(B)}+l(B).\]

Let $l(B)=3$. The same discussion leads to $a_{12}+a_{23}\ge 2$ after a suitable (simultaneous) permutation of rows and columns (i.\,e. replacing $A$ by $P^{\text{T}}AP$ with a permutation matrix $P$). It is not always possible to achieve $\alpha\le\beta\le\gamma$ additionally. But since the trace of $A$ is symmetric in $\alpha$, $\beta$ and $\gamma$, we may assume $2\le\alpha\le\beta\le\gamma$ nevertheless. The inequality in \cite{Barnes} reads 
\[4\alpha\beta\gamma-\alpha\beta^2-\alpha^2\gamma=2\alpha\beta\gamma+\alpha\beta(\gamma-\beta)+\alpha\gamma(\beta-\alpha)\le 4p^d,\]
so that 
\[\alpha+\beta+\gamma\le\alpha+\beta+\frac{4p^d+\alpha\beta^2}{4\alpha\beta-\alpha^2}=:f(\alpha,\beta).\]
We describe a set which contains all feasible points. Since $2\alpha^3\le2\alpha\beta\gamma+\alpha\beta(\gamma-\beta)+\alpha\gamma(\beta-\alpha)\le4p^d$ we get $2\le\alpha\le\sqrt[3]{2p^d}$. Similarly $4\beta^2\le 4p^d$ and $\alpha\le\beta\le\sqrt{p^d}$. Thus all feasible points are contained in the convex polygon
\[\mathcal{F}:=\bigl\{(\alpha,\beta):2\le\alpha\le\sqrt[3]{2p^d},\ \alpha\le\beta\le\sqrt{p^d}\bigr\}.\]
It can be shown (maybe with the help of a computer) that $f$ is convex on $\mathcal{F}$. Hence, the maximal value of $f$ on $\mathcal{F}$ will be attained on one of the $3$ vertices:
\begin{align*}
V_1&=(2,2),\\
V_2&=(2,\sqrt{p^d}),\\
V_3&=(\sqrt[3]{2p^d},\sqrt[3]{2p^d}).
\end{align*}
One can check that $(p^d+14)/3=f(V_1)\ge f(V_2)$ for $p^d\ge10$ and $f(V_1)\ge f(V_3)$ for $p^d\ge 12$. If $p^d\le 10$, then $V_1$ is the only feasible point. In the remaining case $p^d=11$ there is only one more feasible pair $(\alpha,\beta)=(2,3)$. Then $\gamma=3$ and $\alpha+\beta+\gamma\le f(V_1)$.
Now inequality~\eqref{KW} takes the form
\[k(B)\le\alpha+\beta+\gamma-a_{12}-a_{23}\le f(V_1)-2=\frac{p^d-1}{l(B)}+l(B).\]

Finally let $l(B)=4$. By permuting rows and columns and changing signs, we can reach (using \autoref{decomposable}) at least one of the two arrangements
\begin{enumerate}[(i)]
\item $a_{12}+a_{23}+a_{34}\ge3$,\label{case1}
\item $a_{12}+a_{13}+a_{14}\ge3$.\label{case2}
\end{enumerate}
In case \eqref{case1} we can use inequality~\eqref{KW} as before. In case \eqref{case2} we can use \autoref{kulswada} with the quadratic form $q$ corresponding to the positive definite matrix 
\[\frac{1}{2}\begin{pmatrix}2&-1&-1&-1\\-1&2&0&0\\-1&0&2&0\\-1&0&0&2\end{pmatrix}.\]
Thus, for the rest of the proof we will assume that case \eqref{case1} occurs. As before, we will also assume $2\le\alpha\le\beta\le\gamma\le\delta$ and
\begin{equation}\label{Barnes3}
\begin{split}
&4\alpha\beta\gamma\delta-\alpha^2\gamma\delta-\alpha\beta^2\delta-\alpha\beta\gamma^2+\frac{1}{4}\alpha^2(\gamma-\beta)^2\\
&=\alpha\beta\gamma\delta+\alpha\gamma\delta(\beta-\alpha)+\alpha\beta\delta(\gamma-\beta)+\alpha\beta\gamma(\delta-\gamma)+\frac{1}{4}\alpha^2(\gamma-\beta)^2\le 4p^d
\end{split}
\end{equation}
by \cite{Barnes}. We search for the maximum of the function
\[f(\alpha,\beta,\gamma):=\alpha+\beta+\gamma+\frac{4p^d+\alpha\beta\gamma^2-\frac{1}{4}\alpha^2(\gamma-\beta)^2}{4\alpha\beta\gamma-\alpha^2\gamma-\alpha\beta^2}\]
on a suitable convex polyhedron. Since $\alpha^4\le 4p^d$ we have $2\le\alpha\le\sqrt[4]{4p^d}$. In a similar way, we obtain the set
\[\mathcal{F}:=\{(\alpha,\beta,\gamma):2\le\alpha\le\sqrt[4]{4p^d},\ \alpha\le\beta\le\sqrt[3]{2p^d},\ \beta\le\gamma\le\sqrt{p^d}\},\]
which contains all feasible points. It can be shown that $f$ is in fact convex on $\mathcal{F}$. The vertices of $\mathcal{F}$ are
\begin{align*}
V_1&:=(2,2,2),\\
V_2&:=(2,2,\sqrt{p^d}),\\
V_3&:=(2,\sqrt[3]{2p^d},\sqrt[3]{2p^d}),\\
V_4&:=(\sqrt[4]{4p^d},\sqrt[4]{4p^d},\sqrt[4]{4p^d}).
\end{align*}
We fix the value $m:=(p^d+27)/4$.
A calculation shows $f(V_2)\le m$ for $p^d\ge 22$, $f(V_3)\le m$ for $p^d\ge 20$, and $f(V_4)\le m$ for $p^d\ge 23$. If $p^d\le 12$, then $V_1$ is the only feasible point. If $p^d\le 17$, there is only one other feasible point $(\alpha,\beta,\gamma)=(2,2,3)$ beside $V_1$. In this case $f(2,2,3)\le m$ for $p^d\ge 14$. For $p^d=13$ we have 
\[\alpha+\beta+\gamma+\delta-a_{13}-a_{14}-a_{34}\le7=\frac{13-1}{4}+4.\]
For $p^d\le 20$ there is one additional point $(\alpha,\beta,\gamma)=(2,3,3)$, which satisfies $f(2,3,3)\le m$. In the remaining cases there is another additional point $(\alpha,\beta,\gamma)=(3,3,3)$. For this we get $f(3,3,3)\le m$ if $p^d\ge 22$. Since $21$ is no prime power, we can consider $f(V_1)=p^d/4+7$ now. If $p>2$, then $p^d/4$ is no integer. In this case
\[\alpha+\beta+\gamma+\delta-a_{13}-a_{14}-a_{34}\le[f(V_1)]-3=\frac{p^d-1}{4}+4,\]
where $[f(V_1)]$ is the largest integer below $f(V_1)$.
Thus, let us assume $\delta=p^d/4+1$ (and $p=2$). 
With the help of a computer one can show that up to equivalence only the possibility
\begin{equation}\label{threeposs}
A=\begin{pmatrix}2&1&0&-1\\1&2&1&0\\0&1&2&1\\-1&0&1&\delta\end{pmatrix}
\end{equation}
has the right determinant (see also the remark following the proof). By considering the corresponding decomposition matrix, one can easily deduce: 
\[k(B)\le\delta+2\le\frac{p^d-1}{l(B)}+l(B).\]
Now it remains to check, that $f$ does not exceed $m$ on other points of $\mathcal{F}$ (this is necessary, since $f(V_1)>m$). For that, we exclude $V_1$ from $\mathcal{F}$ and form a smaller polyhedron. Since only integral values for $\alpha,\beta,\gamma$ are allowed, we get three new vertices:
\begin{align*}
V_5&:=(2,2,3),\\
V_6&:=(2,3,3),\\
V_7&:=(3,3,3)
\end{align*}
But these points were already considered. This finishes the first part of the proof.
The second part follows easily, since for blocks with cyclic defect groups equality holds.
\end{proof}

Olsson showed the $k(B)$-conjecture in the case $l(B)\le 2$ in general (see Corollary~5 in \cite{OlssonIneq}). For $p=2$ he also proved this in the case $l(B)=3$ (see Corollary~7 in \cite{OlssonIneq}).
In the case $l(B)=5$ there is no inequality like \eqref{Barnes3}. However, one can use the so called “fundamental inequality” of quadratic forms
\[\alpha\beta\gamma\delta\epsilon\le 8p^d\]
(see \cite{Barnes}). Of course, the complexity increases rapidly with $l(B)$. For example, the matrix 
\[A=\begin{pmatrix}2&1&0&1&-1\\1&2&1&1&1\\0&1&2&1&-1\\1&1&1&2&1\\-1&1&-1&1&\epsilon\end{pmatrix}\]
with $\epsilon=p^d/4+9$ ($p=2$) has to be considered. We will demonstrate that such matrices cannot occur. For this let $l:=l(B)$ arbitrary, $a_{ii}=2$, and $a_{i,i+1}=1$ for $i=1,\ldots,l-1$. In the following we will speak of Cartan matrices and decomposition matrices always with respect to an arbitrary basic set. This means we can multiply the generalized decomposition matrix from the left with an orthogonal matrix and from the right with an invertible matrix.

The first two columns of the decomposition matrix $Q$ can be arranged in the form
\[\begin{pmatrix}1&.\\1&1\\.&1\\.&.\\\vdots&\vdots\\.&.\end{pmatrix}.\]
By the orthogonality relations, the first three columns cannot have the form
\[\begin{pmatrix}1&.&\pm1\\1&1&.\\.&1&1\\.&.&.\\\vdots&\vdots&\vdots\\.&.&.\end{pmatrix}\text{ or }\begin{pmatrix}1&.&-1\\1&1&1\\.&1&.\\.&.&.\\\vdots&\vdots&\vdots\\.&.&.\end{pmatrix}.\]
That means they have the form
\[\begin{pmatrix}1&.&.\\1&1&.\\.&1&1\\.&.&1\\.&.&.\\\vdots&\vdots&\vdots\\.&.&.\end{pmatrix}\text{ or }\begin{pmatrix}1&.&.\\1&1&1\\.&1&.\\.&.&1\\.&.&.\\\vdots&\vdots&\vdots\\.&.&.\end{pmatrix}.\]
However, both forms give rise to equivalent matrices $A$. 
Similarly, we may assume that the first $l-1$ columns of $Q$ have the form
\[\begin{pmatrix}1&.&\cdots&.\\1&1&\ddots&\vdots\\.&1&\ddots&.\\.&.&\ddots&1\\[1mm].&.&.&1\\.&.&.&.\\\vdots&\vdots&\vdots&\vdots\\.&.&.&.\end{pmatrix}.\]
(Now one can see that the $5\times 5$ matrix above cannot occur.)
If we add suitable multiples of the first $l-1$ columns to the last column, $Q$ becomes
\[\begin{pmatrix}1&.&\cdots&\cdots&.\\1&1&\ddots&&\vdots\\.&1&\ddots&\ddots&\vdots\\.&.&\ddots&1&.\\.&.&.&1&\ast\\.&.&.&.&\ast\\\vdots&\vdots&\vdots&\vdots&\vdots\\.&.&.&.&\ast\end{pmatrix}.\]
Thus, up to equivalence $A$ has the form
\[\begin{pmatrix}2&1&.&\ldots&.\\1&\ddots&\ddots&\ddots&\vdots\\.&\ddots&\ddots&1&.\\\vdots&\ddots&1&2&a\\.&\cdots&.&a&\epsilon\end{pmatrix}\]
with $a\ge 1$ (notice that this matrix does not have to be reduced).
This yields 
\[\epsilon=\frac{p^d+a^2(l-1)}{l}\hspace{1cm}\text{and}\hspace{1cm}k(B)\le l+\epsilon-a^2=\frac{p^d-a^2}{l}+l\le\frac{p^d-1}{l}+l.\]
It seems likely that this configuration allows the largest value for $k(B)$ in general. 

Fujii gives some sufficient conditions for $\det C=p^d$ in \cite{DetCartan}. We remark that $\det C$ can be determined locally with the notion of lower defect groups.

\section{Brauer's $k(B)$-conjecture for defect groups which are central extensions}

The knowledge of the Cartan matrix implies that $l(B)$ is already known. Since $k(B)-l(B)$ is determined locally, it might seems absurd to bound $k(B)$ in terms of Cartan invariants. 
Instead, it would be more useful if one can apply these bounds to blocks of subsections. In this sense the next lemma is an extension of \autoref{kulswada}.

\begin{Lemma}\label{majorast}
Let $(u,b)$ be a major subsection associated with the block $B$. Let $C_b=(c_{ij})$ be the Cartan matrix of $b$ up to equivalence. Then for every positive definite, integral quadratic form
$q(x_1,\ldots,x_{l(b)})=\sum_{1\le i\le j\le l(b)}{q_{ij}x_ix_j}$
we have
\[k(B)\le\sum_{1\le i\le j\le l(b)}{q_{ij}c_{ij}}.\]
\end{Lemma}
\begin{proof}
Let us consider the generalized decomposition numbers $d^u_{ij}$ associated with the subsection $(u,b)$. 
We write $d_i:=(d_{i1}^u,d_{i2}^u,\ldots,d_{i,l(b)}^u)$ for $i=1,\ldots,k(B)$. Let $Q=(\widetilde{q}_{ij})_{i,j=1}^{l(b)}$ with 
\[\widetilde{q}_{ij}:=\begin{cases}q_{ij}&\text{if }i=j,\\q_{ij}/2&\text{if }i\ne j\end{cases}.\] 
Then we have
\[\sum_{1\le i\le j\le l(b)}{q_{ij}c_{ij}}=\sum_{1\le i\le j\le l(b)}{\sum_{r=1}^{k(B)}{q_{ij}d^u_{ri}\overline{d^u_{rj}}}}=\sum_{r=1}^{k(B)}{d_rQ\overline{d_r}^{\text{T}}},\]
and it suffices to show 
\[\sum_{r=1}^{k(B)}{d_rQ\overline{d_r}^{\text{T}}}\ge k(B).\]
For this, let $p^n$ be the order of $u$. Then $d^u_{ij}$ is an integer of the $p^n$-th cyclotomic field $\mathbb{Q}(\zeta)$ for $\zeta:=e^{2\pi i/p^n}$. It is known that $1,\zeta,\zeta^2,\ldots,\zeta^f$ with $f=p^{n-1}(p-1)-1$ form a basis for the ring of integers of $\mathbb{Q}(\zeta)$. 
We fix a row index $i\in\{1,\ldots,k(B)\}$ of the generalized decomposition matrix. Then there are integers $a_j^0,\ldots,a_j^f\in\mathbb{Z}$ such that $d_{ij}^u:=a_j^0+a_j^1\zeta+\ldots+a_j^f\zeta^f$ for $j=1,\ldots,l(b)$. We define $d:=d_i$ and $a_m:=(a_1^m,a_2^m,\ldots,a_{l(b)}^m)$ for $m=0,\ldots,f$. Since $(u,b)$ is major, at least one of the numbers $a_j^m$ does not vanish. 
Let $\mathcal{G}$ be the Galois group of $\mathbb{Q}(\zeta)$ over $\mathbb{Q}$. Then it is known that for every $\gamma\in\mathcal{G}$ there is an index $i'\in\{1,\ldots,k(B)\}$ such that $\gamma(d):=(\gamma(d_{i1}^u),\ldots,\gamma(d_{i,l(b)}^u))=(d_{i'1}^u,\ldots,d_{i',l(b)}^u)$.
Thus, it suffices to show
\[\sum_{\gamma\in\mathcal{G}}{\gamma(d)Q\overline{\gamma(d)}^{\text{T}}}=\sum_{\gamma\in\mathcal{G}}{\gamma(dQ\overline{d}^\text{T})}\ge|\mathcal{G}|=f+1.\]
We have
\begin{align*}
\sum_{\gamma\in\mathcal{G}}{\gamma(dQ\overline{d}^\text{T})}&=\sum_{\gamma\in\mathcal{G}}{\gamma\Biggl(\sum_{i=0}^f{a_iQa_i^{\text{T}}}+\sum_{j=1}^f{\sum_{m=0}^{f-j}{a_mQa_{m+j}^{\text{T}}(\zeta^j+\overline{\zeta}^j)}}\Biggr)}\\
&=(f+1)\sum_{i=0}^f{a_iQa_i^{\text{T}}}+2\sum_{j=1}^f{\sum_{m=0}^{f-j}{a_mQa_{m+j}^{\text{T}}\sum_{\gamma\in\mathcal{G}}{\gamma(\zeta^j)}}}.
\end{align*}
The $p^m$-th cyclotomic polynomial $\Phi_{p^m}$ has the form
\[\Phi_{p^m}=X^{p^{m-1}(p-1)}+X^{p^{m-1}(p-2)}+\ldots+X^{p^{m-1}}+1.\]
This gives
\[\sum_{\gamma\in\mathcal{G}}{\gamma(\zeta^j)}=\begin{cases}-p^{n-1}&\text{if }p^{n-1}\mid j\\0&\text{else}\end{cases}\]
for $j\in\{1,\ldots,f\}$. It follows that
\begin{align}
\sum_{\gamma\in\mathcal{G}}{\gamma(dQ\overline{d}^\text{T})}&=(f+1)\sum_{i=0}^f{a_iQa_i^{\text{T}}}-2p^{n-1}\sum_{j=1}^{p-2}{\hspace{2mm}\sum_{m=0}^{f-jp^{n-1}}{a_mQa_{m+p^{n-1}j}^{\text{T}}}}\nonumber\\
&=p^{n-1}\Biggl((p-1)\sum_{i=0}^f{a_iQa_i^{\text{T}}}-2\sum_{j=1}^{p-2}{\hspace{2mm}\sum_{m=0}^{f-jp^{n-1}}{a_mQa_{m+p^{n-1}j}^{\text{T}}}}\Biggr).\label{secondsum}
\end{align}
For $p=2$ the claim follows immediately, since then $f+1=2^{n-1}$. Thus, suppose $p>2$. Then we have
\[\bigl\{0,1,\ldots,f-jp^{n-1}\bigr\}\mathbin{\dot{\cup}}\bigl\{(p-1-j)p^{n-1},(p-1-j)p^{n-1}+1,\ldots,f\bigr\}=\{0,1,\ldots,f\}\]
for all $j\in\{1,\ldots,p-2\}$. This shows that every row $a_m$ occurs exactly $p-2$ times in the second sum of \eqref{secondsum}. Hence,
\[\sum_{\gamma\in\mathcal{G}}{\gamma(dQ\overline{d}^\text{T})}=p^{n-1}\Biggl(\sum_{i=0}^f{a_iQa_i^{\text{T}}}+\sum_{j=1}^{p-2}{\hspace{2mm}\sum_{m=0}^{f-jp^{n-1}}{(a_m-a_{m+jp^{n-1}})Q(a_m-a_{m+jp^{n-1}})^{\text{T}}}}\Biggr).\]
Now assume that $a_m$ does not vanish for some $m\in\{0,\ldots,f\}$. Then we have $a_mQa_m^{\text{T}}\ge 1$, since $Q$ is positive definite. 
Again, $a_m$ occurs exactly $p-2$ times in the second sum. Let $a_m-a_{m'}$ (resp. $a_{m'}-a_m$) be such an occurrence. Then we have
\[a_{m'}Qa_{m'}^{\text{T}}+(a_m-a_{m'})Q(a_m-a_{m'})^{\text{T}}\ge 1.\]
Now the claim follows easily.
\end{proof}

The rest of this paper consists of applications of the last lemma mainly in the case $p=2$.
Landrock has shown that Brauer's $k(B)$-conjecture holds for $2$-blocks with defect $3$ (see \cite{Landrock2}). The next theorem generalizes this. 

\begin{Theorem}\label{centext}
Brauer's $k(B)$-conjecture holds for defect groups which are central extensions of metacyclic $2$-groups by cyclic groups. In particular the $k(B)$-conjecture holds for abelian defect $2$-groups of rank at most $3$.
\end{Theorem}
\begin{proof}
Let $B$ be a block with defect group $D$ as in the statement of the theorem. Choose $\langle u\rangle\le\Z(D)$ such that $D/\langle u\rangle$ is metacyclic. For a $B$-subsection $(u,b)$ the block $b$ dominates a block $\overline{b}$ of $\C_G(u)/\langle u\rangle$ with defect group $D/\langle u\rangle$. If $C$ is the Cartan matrix of $\overline{b}$, then $|\langle u\rangle|C$ is the Cartan matrix of $b$.

Hence, by \autoref{majorast} it suffices to show 
\[\sum_{i=1}^{l(B)}{c_{ii}}-\sum_{i=1}^{l(B)-1}{c_{i,i+1}}\le |D|\]
for every $2$-block $B$ with metacyclic defect group $D$ and Cartan matrix $C=(c_{ij})$. If $D$ is dihedral, then $\det C=|D|$ and $l(B)\le 3$ (see \cite{Brauer}). Thus, in this case the claim follows from the proof of \autoref{detCD}. If $D$ is a semidihedral or quaternion group, one can use the tables in \cite{Erdmann} to show the claim (these cases can also be done by the method of the proof of \autoref{detCD} and the fact that the elementary divisors of $C$ are contained in $\{1,2,|D|\}$). 
Now assume $D\cong C_{2^r}\times C_{2^r}=:C_{2^r}^2$ for some $r\in\mathbb{N}$. Then the inertial index $e(B)$ is $1$ or $3$. In case $e(B)=1$, we also have $l(B)=1$, and the claim follows. Thus, we may assume $e(B)=3$. Then by the method of Usami and Puig (see \cite{Usami23I,Usami23II,UsamiZ4}) there is a perfect isometry between $B$ and the principal block of $D\rtimes C_3$. Usami and Puig did not provide an explicit proof for $p=2$, but the author has shown as another part of his PhD thesis that the perfect isometry indeed exists at least in this special case. Using this, we see that $\det C=|D|$ and $l(B)\in\{1,3\}$. Hence, the claim follows as before. By the result of \cite{Sambale}, we are done.
\end{proof}

We note that Brauer has proved the $k(B)$-conjecture for abelian defect groups of rank $2$ and arbitrary primes $p$ (see (7D) in \cite{BrauerBlSec2}). The smallest $2$-group which does not satisfy the hypothesis of \autoref{centext} is the elementary abelian group of order $16$. However, this group can be handled as well.

\begin{Theorem}\label{elem8}
Brauer's $k(B)$-conjecture holds for defect groups which contain a central cyclic subgroup of index at most $8$.
\end{Theorem}
\begin{proof}
Since every group of order $8$ is metacyclic or elementary abelian, it suffices to consider a block $B$ with defect group $D\cong C_2^3$ and Cartan matrix $C=(c_{ij})$. As in \autoref{centext} we show 
\begin{equation}\label{ineq}
\sum_{i=1}^{l(B)}{c_{ii}}-\sum_{i=1}^{l(B)-1}{c_{i,i+1}}\le 8.
\end{equation}
If $e(B)$ is $1$, then also $l(B)=1$, and the claim follows. 

Now let $e(B)=3$. (This case can be handled easily with the method of Usami and Puig. Due to the lack of an explicit proof in the case $p=2$ as said before, we do not use their method here.)
It is easy to show that there are four subsections $(1,B)$, $(u_1,b_1)$, $(u_2,b_2)$ and $(u_3,b_3)$ associated with $B$. Moreover, we may assume $l(b_1)=3$ and $l(b_2)=l(b_3)=1$. As usual, $b_1$ dominates a block of $\operatorname{C}_G(u_1)/\langle u_1\rangle$ with Klein four defect group. It follows that the Cartan matrix of $b_1$ is equivalent to 
\[\begin{pmatrix}4&2&2\\2&4&2\\2&2&4\end{pmatrix}.\]
Using this, is it easy to see that there is a basic set such that the generalized decomposition numbers associated with $u_i$ ($i=1,2,3$) have form
\[\begin{pmatrix}1&.&.&1&1\\1&.&.&1&-1\\1&1&.&-1&1\\1&1&.&-1&-1\\.&1&1&1&1\\.&1&1&1&-1\\.&.&1&-1&1\\.&.&1&-1&-1\end{pmatrix}.\]
By the orthogonality relations of generalized decomposition numbers there exists a matrix $S\in\operatorname{GL}(3,\mathbb{Q})$ such that the ordinary decomposition matrix $Q$ satisfies
\[Q=\begin{pmatrix}1&.&.\\-1&.&.\\.&1&.\\.&-1&.\\.&.&1\\.&.&-1\\-1&-1&-1\\1&1&1\end{pmatrix}S\]
Moreover, it is easy to see that all entries of $S$ are integral. It is well known that there exists a matrix $\widetilde{Q}\in\mathbb{Z}^{3\times 8}$ such that $\widetilde{Q}Q=1_3$. This shows $S\in\operatorname{GL}(3,\mathbb{Z})$. 
Hence $C$ has the form $S^{-\text{T}}Q^{\text{T}}QS^{-1}$ up to equivalence. 
Thus, the claim follows in this case.

Let $e(B)=7$. Then there are two subsections $(1,B)$ and $(u,b)$ with $k(B)-l(B)=l(b)=1$. Since $8$ is the sum of $k(B)$ integer squares, we must have $k(B)\in\{5,8\}$. By Corollary~1 in \cite{DetCartan}, we have $\det C=8$. Thus in the case $l(B)=4$, the claim follows from the proof of \autoref{detCD} (notice that this case contradicts Brauer's height zero conjecture). 
So we may assume $l(B)=7$. Then the generalized decomposition numbers corresponding to $u$ can be arranged in the form $(1,\ldots,1)^{\text{T}}$. Hence the ordinary decomposition matrix has the form
\[\begin{pmatrix}1&.&.&.&.&.&.\\-1&-1&.&.&.&.&.\\.&1&1&.&.&.&.\\.&.&-1&-1&.&.&.\\.&.&.&1&1&.&.\\.&.&.&.&-1&-1&.\\.&.&.&.&.&1&1\\.&.&.&.&.&.&-1\end{pmatrix},\]
and the claim follows.

Let $e(B)=21$. Then there are two subsections $(1,B)$ and $(u,b)$ with $k(B)-l(B)=l(b)=3$. In particular $l(B)\le 5$ (using \autoref{centext}). The theory of lower defect groups reveals that $2$ occurs at least twice as elementary divisor of $C$. This gives $l(B)\ge 3$. The case $l(B)=3$ contradicts Corollary~1.3 in \cite{Landrock2}. Now let $l(B)=4$ (again this case contradicts the height zero conjecture). Then the generalized decomposition numbers corresponding to $u$ have the form
\[\begin{pmatrix}1&.&1\\1&.&.\\1&1&1\\1&1&.\\.&1&1\\.&1&.\\.&.&1\end{pmatrix}.\]
That means the ordinary decomposition matrix becomes
\[\begin{pmatrix}1&.&.&.\\-1&-1&.&.\\.&.&-1&.\\.&1&1&.\\.&.&.&-1\\.&-1&.&1\\-1&.&1&1\end{pmatrix},\]
and the Cartan matrix has the form
\[C=\begin{pmatrix}3&1&-1&-1\\1&3&1&-1\\-1&1&3&1\\-1&-1&1&3\end{pmatrix}.\]
Unfortunately, this matrix does not satisfy inequality \eqref{ineq}. However, we can use \autoref{majorast} with the quadratic form $q$ corresponding to the positive definite matrix
\[\frac{1}{2}\begin{pmatrix}2&-1&1&.\\-1&2&-1&.\\1&-1&2&-1\\.&.&-1&2\end{pmatrix}.\]

Finally let $l(B)=5$. Then the generalized decomposition numbers corresponding to $u$ have the form
\[\begin{pmatrix}1&.&.\\1&.&.\\1&1&.\\1&1&.\\.&1&1\\.&1&1\\.&.&1\\.&.&1\end{pmatrix}\]
and the ordinary decomposition matrix becomes
\[\begin{pmatrix}1&.&.&.&.\\-1&.&.&.&-1\\.&1&.&.&1\\.&-1&.&.&.\\.&.&-1&.&-1\\.&.&1&.&.\\.&.&.&1&1\\.&.&.&-1&.\end{pmatrix}.\]
Thus, the Cartan matrix is
\[\begin{pmatrix}2&.&.&.&1\\.&2&.&.&1\\.&.&2&.&1\\.&.&.&2&1\\1&1&1&1&4\end{pmatrix}.\]
In this case we can use \autoref{majorast} with the quadratic form $q$ corresponding to the positive definite matrix
\[\frac{1}{2}\begin{pmatrix}2&1&.&.&-1\\1&2&.&.&-1\\.&.&2&.&-1\\.&.&.&2&-1\\-1&-1&-1&-1&2\end{pmatrix}.\qedhere\]
\end{proof}

Recently, Kessar, Koshitani and Linckelmann have proven that the cases $k(B)=5$ and $k(B)=7$ in the proof above cannot occur (see \cite{KKL}). However, their proof uses the classification of finite simple groups. 

We deduce a corollary.

\begin{Corollary}\label{defect4}
Brauer's $k(B)$-conjecture holds for $2$-blocks of defect at most $4$.
\end{Corollary}

We note that Robinson showed $k(B)\le 2^5$ for every $2$-block of defect $4$ (see \cite{RobBrauerFeit}). 
For odd primes it is only known that the $k(B)$-conjecture holds for blocks of defect at most $2$. We improve this in the case $p=3$.

\begin{Theorem}
Brauer's $k(B)$-conjecture holds for defect groups which contain a central cyclic subgroup of index at most $9$.
\end{Theorem}
\begin{proof}
It suffices to consider blocks $B$ with elementary abelian defect groups $D$ of order $9$. For this, we use the work \cite{Kiyota} by Kiyota. We have $e(B)\in\{1,2,4,8,16\}$. As usual, we may assume $e(B)>1$. We denote the Cartan matrix of $B$ by $C$.

\textbf{Case 1:} $e(B)=2$.\\
By \cite{Usami23I} we may assume that $G=D\rtimes C_2$ (observe that there are two essentially different actions of $C_2$ on $D$).
It is easy to show that $D$ is given by
\[\begin{pmatrix}
5&4\\4&5
\end{pmatrix}\text{ or }\begin{pmatrix}
6&3\\3&6
\end{pmatrix}.\]

\textbf{Case 2:} $e(B)=4$.\\
If the inertial group $I(B)$ is cyclic, we obtain $C$ up to equivalence as follows
\[\begin{pmatrix}
3&2&2&2\\2&3&2&2\\2&2&3&2\\2&2&2&3
\end{pmatrix}\]
from \cite{UsamiZ4}. If $I(B)$ is noncyclic, we have to deal with twisted group algebras of $D\rtimes C_2^2$ as in \cite{UsamiZ2Z2}. Let $\gamma$ be the corresponding $2$-cocycle. Then there are just two possibilities for $\gamma$. In particular there are at most two equivalence classes for $C$. If $\gamma$ is trivial, the $C$ is equivalent to
\[\begin{pmatrix}
4&1&2&2\\1&4&2&2\\2&2&4&2\\2&2&1&4
\end{pmatrix}.\]
In the other case Kiyota gives the following example: Let $Q_8$ act on $D$ with kernel $\Z(Q_8)$. Then we can take the nonprincipal block of $D\rtimes Q_8$ for $B$. In this case $l(B)=1$, so the claim follows.

\textbf{Case 3:} $I(B)\cong C_8$.\\
Then $I(B)$ acts regularly on $D\setminus\{1\}$. Thus, there are just two $B$-subsections $(1,B)$ and $(u,b)$ with $l(b)=1$. Kiyota did not obtain the block invariants in this case. Hence, we have to consider some possibilities. By Lemma~(1D) in \cite{Kiyota} we have $k(B)\in\{3,6,9\}$. Since $u$ is conjugate to $u^{-1}$ in $I(B)$, the generalized decomposition numbers $d^u_{ij}$ are integers. Suppose $k(B)=3$. Then the column corresponding to $(u,b)$ in the generalized decomposition matrix has the form $(\pm2,\pm2,\pm1)^{\text{T}}$. Hence, $C$ is equivalent to
\[\begin{pmatrix} 5&1\\1&2\end{pmatrix}.\]
In the case $k(B)=6$ the column corresponding to $(u,b)$ is given by $(\pm2,\pm1,\pm1,\pm1,\pm1,\pm1)^{\text{T}}$, and $C$ is equivalent to
\[\begin{pmatrix}
2&1&1&1&0\\
1&2&1&1&1\\
1&1&2&1&1\\
1&1&1&2&1\\
0&1&1&1&3
\end{pmatrix}.\]
Finally in the case $k(B)=9$ we get the following Cartan matrix:
\[\begin{pmatrix}
2&1&1&1&1&1&1&1\\
1&2&1&1&1&1&1&1\\
1&1&2&1&1&1&1&1\\
1&1&1&2&1&1&1&1\\
1&1&1&1&2&1&1&1\\
1&1&1&1&1&2&1&1\\
1&1&1&1&1&1&2&1\\
1&1&1&1&1&1&1&2
\end{pmatrix}.\]

\textbf{Case 4:} $I(B)\cong D_8$.\\
By Proposition~(2F) in \cite{Kiyota} there are two possibilities: $(k(B),l(B))\in\{(9,5),(6,2)\}$. In both cases there are three subsections $(1,B)$, $(u_1,b_1)$ and $(u_2,b_2)$ with $l(b_1)=l(b_2)=2$. The Cartan matrix of $b_1$ and $b_2$ is given by $\bigl(\begin{smallmatrix}6&3\\3&6\end{smallmatrix}\bigr)$. In the case $k(B)=9$ and $l(B)=5$ the numbers $d^{u_1}_{ij}$ and $d^{u_2}_{ij}$ are integers (see Subcase~(a) on page 39 in \cite{Kiyota}). Thus, we may assume that the numbers $d^{u_1}_{ij}$ form the two columns
\[\begin{pmatrix}
1&1&1&1&1&1&.&.&.\\
.&.&.&1&1&1&1&1&1
\end{pmatrix}^{\text{T}}.\]
Now we use a GAP program to enumerate the possibilities for the columns $(d^{u_2}_{1j},d^{u_2}_{2j},\ldots,d^{u_2}_{9j})$ ($j=1,2$). It turns out that $C$ is equivalent to
\[\begin{pmatrix}
3&.&1&.&1\\
.&3&1&.&1\\
1&1&3&1&.\\
.&.&1&3&1\\
1&1&.&1&3
\end{pmatrix}\]
in all cases. Here we can take the quadratic form $q$ corresponding to the matrix
\[\frac{1}{2}\begin{pmatrix}
2&.&-1&.&-1\\
.&2&-1&1&-1\\
-1&-1&2&-1&1\\
.&1&-1&2&-1\\
-1&-1&1&-1&2
\end{pmatrix}\]
in \autoref{majorast}.

In the case $k(B)=6$ and $l(B)=2$ the columns $d_1:=(d^{u_1}_{11},d^{u_1}_{21},\ldots,d^{u_1}_{61})$ and $d_2:=(d^{u_1}_{12},d^{u_1}_{22},\ldots,d^{u_1}_{62})$ do not consist of integers only. We write $d_1=a+b\zeta$ with $a,b\in\mathbb{Z}^6$ and $\zeta:=e^{2\pi i/3}$. Then $d_2=a+b\overline{\zeta}$. The orthogonality relations show that
\begin{align*}
6&=(d_1\mid d_1)=(a\mid a)+(b\mid b)-(a\mid b),\\
3&=(d_1\mid d_2)=(a\mid a)+2(a\mid b)\zeta+(b\mid b)\overline{\zeta}=(a\mid a)-(b\mid b)+(2(a\mid b)-(b\mid b))\zeta.
\end{align*}
This shows $(a\mid a)=5$, $(b\mid b)=2$ and $(a\mid b)=1$. Hence, we can arrange $d_1$ in the following way: \[(1,1,1,1,1+\zeta,1+\overline{\zeta}=-\zeta)^{\text{T}}.\] 
It is easy to see that there are essentially two possibilities for the column $(d^{u_2}_{11},d^{u_2}_{21},\ldots,d^{u_2}_{61})^{\text{T}}$: 
\[(1+\zeta,-\zeta,-1,-1,1,1)^{\text{T}}\text{ or }(1+\zeta,-\zeta,-1,1,-1,-1)^{\text{T}}.\] 
The second possibility is impossible, since then $C$ would have determinant $81$. Thus, the first possibility occurs, and $C$ is
\[\begin{pmatrix}
5&1\\1&2
\end{pmatrix}\]
up to equivalence.

\textbf{Case 5:} $I(B)\cong Q_8$.\\
Then $I(B)$ acts regularly on $D\setminus\{1\}$. Hence, the result follows as in the case $I(B)\cong C_8$.

\textbf{Case 6:} $e(B)=16$.\\
Then there are two $B$-subsections $(1,B)$ and $(u,b)$ up to conjugation. This time we have $l(b)=2$. By Proposition~(2G) in \cite{Kiyota} and Watanabe we have $k(B)=9$ and $l(B)=7$. The Cartan matrix of $b$ is given by $\bigl(\begin{smallmatrix}6&3\\3&6\end{smallmatrix}\bigr)$. 
By way of contradiction, we assume that the columns $d_1:=(d^u_{11},d^u_{21},\ldots,d^u_{91})$ and $d_2:=(d^u_{12},d^u_{22},\ldots,d^u_{92})$ are $3$-conjugate. Then an argument as in Case~4 shows the contradiction $k(B)\le6$. Hence, the columns $d_1$ and $d_2$ have the form
\[\begin{pmatrix}
1&1&1&1&1&1&.&.&.\\
.&.&.&1&1&1&1&1&1
\end{pmatrix}^{\text{T}}.\]
Thus, we obtain $C$ as follows:
\[\begin{pmatrix}
2&1&.&.&.&.&1\\
1&2&.&.&.&.&1\\
.&.&2&1&.&.&1\\
.&.&1&2&.&.&1\\
.&.&.&.&2&1&1\\
.&.&.&.&1&2&1\\
1&1&1&1&1&1&3
\end{pmatrix}.\]
In this case we can take the quadratic form $q$ corresponding to the matrix
\[\frac{1}{2}\begin{pmatrix}
2&-1&.&.&.&.&-1\\
-1&2&.&.&.&.&.\\
.&.&2&-1&.&.&-1\\
.&.&-1&2&.&1&.\\
.&.&.&.&2&-1&-1\\
.&.&.&1&-1&2&.\\
-1&.&-1&.&-1&.&2
\end{pmatrix}\]
in \autoref{majorast}.
\end{proof}

\begin{Corollary}
Brauer's $k(B)$-conjecture holds for $3$-blocks of defect at most $3$.
\end{Corollary}

Hendren obtained some informtation about blocks with nonabelian defect groups of order $p^3$ (see \cite{Hendren1,Hendren2}). In particular he showed that Brauer's $k(B)$-conjecture is satisfied in the exponent $p^2$ case. However, he was not able to prove this in the exponent $p$ case, even for $p=3$ (see section~6.1 in \cite{Hendren1}).

We add two other results in the same spirit.

\begin{Proposition}
Brauer's $k(B)$-conjecture holds for defect groups which are central extensions of $C_4\wr C_2$ by a cyclic group.
\end{Proposition}
\begin{proof}
This follows from section~IV in \cite{Kuelshammerwr}.
\end{proof}

\begin{Theorem}
Let $Q$ be a minimal nonabelian $2$-group, but not of type $\langle x,y\mid x^{2^r}=y^{2^r}=[x,y]^2=[x,x,y]=[y,x,y]=1\rangle$ with $r\ge 3$, $[x,y]:=xyx^{-1}y^{-1}$ and $[x,x,y]:=[x,[x,y]]$ (these groups have order $2^{2r+1}\ge 128$). Then Brauer's $k(B)$-conjecture holds for defect groups which are central extensions of $Q$ by a cyclic group. Moreover, the $k(B)$-conjecture holds for all $2$-blocks with minimal nonabelian defect groups.
\end{Theorem}
\begin{proof}
This follows from another part of the author's PhD thesis (see \cite{Sambalemna}).
\end{proof}

\section{A counterexample}\label{countersection}

Külshammer and Wada have shown that there is not always a quadratic form $q$ such that we have equality in \autoref{kulswada}. However, it is not clear if there is always a quadratic form $q$ such that
\begin{equation}\label{countereq}
\sum_{1\le i\le j\le l(B)}{q_{ij}c_{ij}}\le p^d,
\end{equation}
where $d$ is the defect of the block $B$.
(Of course, this would imply the $k(B)$-conjecture in general.)

We consider an example. Let $D\cong C_2^4$, $A\in\operatorname{Syl}_3(\Aut(D))$, $G=D\rtimes A$ and $B=B_0(G)$. Then $k(B)=16$, $l(B)=9$, and the decomposition matrix $Q$ and the Cartan matrix $C$ of $B$ are 
\[Q=\left(\begin{array}{ccccccccc}
1&.&.&.&.&.&.&.&.\\
.&1&.&.&.&.&.&.&.\\
.&.&1&.&.&.&.&.&.\\
.&.&.&1&.&.&.&.&.\\
.&.&.&.&1&.&.&.&.\\
.&.&.&.&.&1&.&.&.\\
.&.&.&.&.&.&1&.&.\\
.&.&.&.&.&.&.&1&.\\
.&.&.&.&.&.&.&.&1\\
1&1&1&.&.&.&.&.&.\\
1&.&.&.&.&1&1&.&.\\
.&.&.&1&.&1&.&1&.\\
.&.&.&.&1&.&1&.&1\\
.&1&.&.&1&.&.&1&.\\
.&.&1&1&.&.&.&.&1\\
1&1&1&1&1&1&1&1&1
\end{array}\right),\ C=\left(\begin{array}{ccccccccc}
4&2&2&1&1&2&2&1&1\\
2&4&2&1&2&1&1&2&1\\
2&2&4&2&1&1&1&1&2\\
1&1&2&4&1&2&1&2&2\\
1&2&1&1&4&1&2&2&2\\
2&1&1&2&1&4&2&2&1\\
2&1&1&1&2&2&4&1&2\\
1&2&1&2&2&2&1&4&1\\
1&1&2&2&2&1&2&1&4
\end{array}\right).
\]
We will see that in this case there is no quadratic form $q$ such that inequality~\eqref{countereq} is satisfied. In order to do so, we assume that $q$ is given by the matrix $\frac{1}{2}A$ with $A=(a_{ij})\in\mathbb{Z}^{9\times 9}$. Since $A$ is symmetric, we only consider the upper triangular half of $A$. Then the rows of $Q$ are $1$-roots of $q$, i.\,e. $rAr^{\text{T}}=2$ for every row $r$ of $Q$ (see Corollary~B in \cite{KuelshammerWada}). If we take the first nine rows of $Q$, it follows that $a_{ii}=2$ for $i=1,\ldots,9$. Now assume $|a_{12}|\ge 2$. Then 
\[(1,-\operatorname{sgn}a_{12},0,\ldots,0)A(1,-\operatorname{sgn}a_{12},0,\ldots,0)^{\text{T}}\le 0,\] 
and $q$ is not positive definite. The same argument shows $a_{ij}\in\{-1,0,1\}$ for $i\ne j$. In particular there are only finitely many possibilities for $q$. Now the next row of $Q$ shows 
\[(a_{12},a_{13},a_{23})\in\{(-1,-1,0),(-1,0,-1),(0,-1,-1)\}.\] 
The same holds for the following triples 
\[(a_{16},a_{17},a_{67}),\ (a_{46},a_{48},a_{68}),\ (a_{57},a_{59},a_{79}),\ (a_{25},a_{28},a_{58}),\ (a_{34},a_{39},a_{49}).\] 
Finally the last row of $Q$ shows that the remaining entries add up to $4$:
\[a_{14}+a_{15}+a_{18}+a_{19}+a_{24}+a_{26}+a_{27}+a_{29}+a_{35}+a_{36}+a_{37}+a_{38}+a_{45}+a_{47}+a_{56}+a_{69}+a_{78}+a_{89}=4.\]
These are too many possibilities to check by hand. So we try to find a positive definite form $q$ with GAP. To decrease the computational effort, we enumerate all positive definite $7\times 7$ left upper submatrices of $A$ first. There are $140428$ of them, but none can be completed to a positive definite $9\times9$ matrix with the given constraints.

Nevertheless, we show that there is no corresponding decomposition matrix for $C$ with more than $16$ rows. For this let $B$ be a block with Cartan matrix equivalent to $C$. (By \autoref{defect4} the $k(B)$-conjecture already holds for $B$. We give an independent argument for this.) We enumerate the possible decomposition matrices $Q$ and count their rows. Since $Q\in\mathbb{Z}^{k(B)\times 9}$, every column of $Q$ has the form $(\pm1,\pm1,\pm1,\pm1,0,\ldots,0)^{\text{T}}$ for a suitable arrangement. Let us assume that the first two columns of $Q$ have the form
\[\begin{pmatrix}
1&1&1&1&0&\cdots&0\\
1&1&1&-1&0&\cdots&0
\end{pmatrix}^{\text{T}}.\]
Then the entries of $C$ show that there is no possibility for the fifth column of $Q$. Thus, we may assume that the first two columns of $Q$ are
\[\begin{pmatrix}
1&1&1&1&0&0&0&\cdots&0\\
0&0&1&1&1&1&0&\cdots&0
\end{pmatrix}^{\text{T}}.\]
Now we use a backtracking algorithm with GAP to show that $Q$ has at most $16$ rows.

Unfortunately, this method does not carry over to major subsections. For if we multiply $C$ by a $2$-power (namely the order of a $2$-element), the corresponding (generalized) decomposition matrices can be entirely different. 

\section{$2$-blocks with defect $5$}

In order to proof Brauer's $k(B)$-conjecture for $2$-blocks of defect $5$, we discuss the central extensions of groups of order $16$ by cyclic groups. We start with the abelian (and nonmetacyclic) groups of order $16$. In the next proposition we have to exclude one case, as the last section has shown. Moreover, we use the work of Kessar, Koshitani and Linckelmann (and thus the classification) in the proof. We have not checked if it is possible to avoid the classification by considering more (virtually impossible) cases. For this reason, we will also freely use the method of Usami and Puig.

\begin{Proposition}\label{elem16}
Let $B$ be a block with a defect group which is a central extension of an elementary abelian group of order $16$ by a cyclic group. If $9\nmid e(B)$, then Brauer's $k(B)$-conjecture holds for $B$.
\end{Proposition}
\begin{proof}
Let $D$ be the defect group of $B$. We choose $u\in\Z(D)$ such that $D/\langle u\rangle$ is elementary abelian of order $16$. Let $(u,b)$ be a $B$-subsection. Then it is easy to see that $e(b)$ is a divisor of $e(B)$. By hypothesis $e(b)\in\{1,3,5,7,15,21\}$. 
As in the proof of \autoref{elem8}, we replace $b$ by $B$ for simplicity. In order to proof the Proposition, we determine the Cartan matrix $C$ of $B$ up to equivalence. If this is done, it will be immediately clear that a suitable inequality as in \autoref{majorast} is satisfied. 

The case $e(B)=1$ is clear. 
We consider the remaining cases separately.

\textbf{Case 1:} $e(B)=3$.\\
In this case we may use the method of Usami and Puig (see \cite{Usami23I,Usami23II,UsamiZ4}). Thus, we can replace $G$ by $D\rtimes C_3$ via a perfect isometry (observe that there are two essentially different actions of $C_3$ on $D$). Then $C$ has the form
\[\begin{pmatrix}8&4&4\\4&8&4\\4&4&8\end{pmatrix}\text{ or }\begin{pmatrix}6&5&5\\5&6&5\\5&5&6\end{pmatrix}\]
up to equivalence.

\textbf{Case 2:} $e(B)=5$.\\
Then there are four subsections $(1,B)$, $(u_1,b_1)$, $(u_2,b_2)$ and $(u_3,b_3)$ with $l(b_1)=l(b_2)=l(b_3)=1$. According to the fact that $|D|=16$ is a sum of $k(B)$ squares, we have six possibilities:
\begin{enumerate}[(i)]
\item $k(B)=k_0(B)=16$ and $l(B)=13$,\label{pos1}
\item $k(B)=k_0(B)=8$ and $l(B)=5$,\label{pos2}
\item $k(B)=13$, $k_0(B)=12$, $k_1(B)=1$ and $l(B)=10$,\label{pos3}
\item $k(B)=10$, $k_0(B)=8$, $k_1(B)=2$ and $l(B)=7$,\label{pos4}
\item $k(B)=7$, $k_0(B)=4$, $k_1(B)=3$ and $l(B)=4$,\label{pos5}
\item $k(B)=5$, $k_0(B)=4$, $k_1(B)=1$ and $l(B)=2$.\label{pos6}
\end{enumerate}
(Brauer's height zero conjecture would contradict the last four possibilities.) In case~\eqref{pos1} we have
\[C=\left(\begin{array}{ccccccccccccc}4&3&3&3&1&1&1&1&1&1&-1&-1&-1\\
3&4&3&3&1&1&1&1&1&1&-1&-1&-1\\
3&3&4&3&1&1&1&1&1&1&-1&-1&-1\\
3&3&3&4&1&1&1&1&1&1&-1&-1&-1\\
1&1&1&1&2&1&1&.&.&.&.&.&.\\
1&1&1&1&1&2&1&.&.&.&.&.&.\\
1&1&1&1&1&1&2&.&.&.&.&.&.\\
1&1&1&1&.&.&.&2&1&1&.&.&.\\
1&1&1&1&.&.&.&1&2&1&.&.&.\\
1&1&1&1&.&.&.&1&1&2&.&.&.\\
-1&-1&-1&-1&.&.&.&.&.&.&2&1&1\\
-1&-1&-1&-1&.&.&.&.&.&.&1&2&1\\
-1&-1&-1&-1&.&.&.&.&.&.&1&1&2\end{array}\right)\]
up to equivalence. In particular $\det C=256$. However, this contradicts Corollary~1 in \cite{DetCartan}. Now we assume that case~\eqref{pos2} occurs. Then there are several ways to arrange the generalized decomposition numbers corresponding to $b_i$ for $i=1,2,3$:
\[\begin{pmatrix}1&-1&-1\\1&-1&-1\\1&-1&-1\\1&-1&-1\\1&-1&-1\\1&-1&3\\1&3&-1\\3&1&1\end{pmatrix},\begin{pmatrix}1&-1&1\\1&-1&1\\1&-1&-1\\1&-1&-1\\1&-1&-1\\1&-1&3\\1&3&1\\3&1&-1\end{pmatrix},\begin{pmatrix}1&1&1\\1&1&1\\1&1&1\\1&-1&-1\\1&-1&-1\\1&-1&3\\1&3&-1\\3&-1&-1\end{pmatrix}.\]
In the last two cases the determinant of $C$ would be $64$. Thus, only the first case can occur. Then we have
\[C=\begin{pmatrix}4&3&3&3&3\\
3&4&3&3&3\\
3&3&4&3&3\\
3&3&3&4&3\\
3&3&3&3&4\end{pmatrix}\]
up to equivalence. Hence, we can consider the case~\eqref{pos3}. Then the generalized decomposition numbers corresponding to $b_i$ for $i=1,2,3$ can be arranged in the form
\[\left(\begin{array}{ccccccccccccc}1&1&1&1&1&1&1&1&1&1&1&1&2\\1&1&1&1&-1&-1&-1&-1&-1&-1&-1&-1&2\\-1&-1&-1&-1&-1&-1&-1&-1&1&1&1&1&2\end{array}\right)^{\text{T}}.\]
However, in this case $C$ would have determinat $256$. In the same manner we see that also the case~\eqref{pos4} is not possible. Thus, assume case~\eqref{pos5}. Then the generalized decomposition numbers corresponding to $b_i$ for $i=1,2,3$ have the form
\[\begin{pmatrix}1&1&1&1&2&2&2\\-1&-1&-1&-1&2&2&-2\\1&1&1&1&2&-2&-2\end{pmatrix}^{\text{T}}.\]
This gives
\[C=\begin{pmatrix}5&4&4&5\\
4&5&4&5\\
4&4&5&5\\
5&5&5&7\end{pmatrix},\]
and the claim follows. Finally let case~\eqref{pos6} occur. Then the generalized decomposition numbers corresponding to $b_i$ for $i=1,2,3$ have the form
\[\begin{pmatrix}1&1&1&3&2\\1&1&-3&-1&2\\1&-3&1&-1&2\end{pmatrix}^{\text{T}}.\]
It follows that
\[C=\begin{pmatrix}4&6\\6&13\end{pmatrix}.\]

\textbf{Case 3:} $e(B)=7$.\\
There are again four subsections $(1,B)$, $(u_1,b_1)$, $(u_2,b_2)$ and $(u_3,b_3)$. But in this case $l(b_1)=l(b_2)=1$ and $l(b_3)=7$ by the Kessar-Koshitani-Linckelmann paper. Moreover, $2$ appears six times as elementary divisor of the Cartan matrix of $b_3$. Using the theory of lower defect groups it follows that $2$ occurs at least six times as elementary divisor of $C$. By \autoref{defect4} we have $k(B)\le 16$. This gives $k(B)=k_0(B)=16$, $l(B)=7$. The generalized decomposition matrix (without the ordinary part) can be arranged in the form
\[\left(\begin{array}{cccccccccccccccc}
1&1&1&1&.&.&.&.&.&.&.&.&.&.&.&.\\
.&.&1&1&1&1&.&.&.&.&.&.&.&.&.&.\\
.&.&.&.&1&1&1&1&.&.&.&.&.&.&.&.\\
.&.&.&.&.&.&1&1&1&1&.&.&.&.&.&.\\
.&.&.&.&.&.&.&.&1&1&1&1&.&.&.&.\\
.&.&.&.&.&.&.&.&.&.&1&1&1&1&.&.\\
.&.&.&.&.&.&.&.&.&.&.&.&1&1&1&1\\
1&1&-1&-1&1&1&-1&-1&1&1&-1&-1&1&1&-1&-1\\
1&-1&1&-1&1&-1&1&-1&1&-1&1&-1&1&-1&1&-1
\end{array}\right)^{\text{T}}.\]
Hence, $C$ has the form
\[\begin{pmatrix}
4&2&2&2&2&2&2\\
2&4&2&2&2&2&2\\
2&2&4&2&2&2&2\\
2&2&2&4&2&2&2\\
2&2&2&2&4&2&2\\
2&2&2&2&2&4&2\\
2&2&2&2&2&2&4
\end{pmatrix}\]
up to equivalence (notice that this is also the Cartan matrix of $b_3$).

\textbf{Case 4:} $e(B)=15$.\\
There are just two subsections $(1,B)$ and $(u,b)$ with $l(b)=1$. It is easy to prove the claim using a similar case decision as in Case~2. We skip the details.

\textbf{Case 5:} $e(B)=21$.\\
There are four subsections $(1,B)$, $(u_1,b_1)$, $(u_2,b_2)$ and $(u_3,b_3)$. We have $l(b_1)=l(b_2)=3$ and $l(b_3)=5$ by the Kessar-Koshitani-Linckelmann paper. With the notations of \cite{KuelshammerOkuyama} $B$ is a centrally controlled block. In particular $l(B)\ge l(b_3)=5$ (see Theorem~1.1 in \cite{KuelshammerOkuyama}). Since the $k(B)$-conjecture holds for $B$, we have $k(B)=16$ and $l(B)=5$. The Cartan matrix of $b_3$ is given in the proof of \autoref{elem8}. Using this it is easy to deduce that the generalized decomposition numbers corresponding to $(u_3,b_3)$ can be arranged in the form
\[\left(\begin{array}{cccccccccccccccc}
1&1&1&1&.&.&.&.&.&.&.&.&.&.&.&.\\
.&.&.&.&1&1&1&1&.&.&.&.&.&.&.&.\\
.&.&.&.&.&.&.&.&1&1&1&1&.&.&.&.\\
.&.&.&.&.&.&.&.&.&.&.&.&1&1&1&1\\
.&.&1&1&.&.&1&1&.&.&1&1&.&.&1&1 
\end{array}\right)^{\text{T}}.\]
It is also easy to see that the columns of generalized decomposition numbers corresponding to $b_1$ and $b_2$ consist of eight entries $\pm1$ and eight entries $0$. The theory of lower defect groups shows that $2$ occurs as elementary divisor of $C$. Now we use GAP to enumerate all possible arrangements of these columns. It turns out that $C$ is equivalent to the Cartan matrix of $b_3$.
\end{proof}

\begin{Proposition}\label{C4C2C2}
Brauer's $k(B)$-conjecture holds for defect groups which are central extensions of $C_4\times C_2^2$ by a cyclic group.
\end{Proposition}
\begin{proof}
Let $B$ be a block with defect group $D\cong C_4\times C_2^2$. We may assume $e(B)=3$. Then we can use the method of Usami and Puig (see \cite{Usami23I,Usami23II,UsamiZ4}). This means it suffices to consider the case $G=D\rtimes C_3$ and $B=B_0(G)$. An easy calculation shows that the Cartan matrix of $B$ is equivalent to
\[\begin{pmatrix}8&4&4\\4&8&4\\4&4&8\end{pmatrix}.\]
Hence, the result follows from inequality~\eqref{KW} as before.
\end{proof}

Now we turn to the nonabelian (and nonmetacyclic) groups of order $16$.

\begin{Proposition}\label{d8xc2}
Let $B$ be a nonnilpotent block with defect group $D_8\times C_2$. Then $k(B)=10$, $k_0(B)=8$ and $k_1(B)=2$. The ordinary irreducible characters are $2$-rational. Moreover, $l(B)\in\{2,3\}$ and the Cartan matrix of $B$ is equivalent to 
\[\begin{pmatrix}6&2\\2&6\end{pmatrix}\text{ or }\begin{pmatrix}6&2&2\\2&4&0\\2&0&4\end{pmatrix}.\]
In particular the $k(B)$-conjecture holds for defect groups which are central extensions of $D_8\times C_2$ by a cyclic group.
\end{Proposition}
\begin{proof}
First we remark that the proof and the result is very similar to the case where the defect group is $D_8$ (see \cite{Brauer}).
Let $D:=\langle x,y,z\mid x^4=y^2=z^2=[x,z]=[y,z]=1,\ yxy=x^{-1}\rangle\cong D_8\times C_2$ and let $(D,b_D)$ a Sylow subpair.
It is easy to see that $\Aut(D)$ is a $2$-group. Thus, $e(B)=1$. We use the theory developed in \cite{Olssonsubpairs}. One can show, that all selfcentralizing proper subgroups of $D$ are maximal and there are precisely three of them:
\begin{align*}
M_1&:=\langle x^2,y,z\rangle\cong C_2^3,\\
M_2&:=\langle x^2,xy,z\rangle\cong C_2^3,\\
M_3&:=\langle x,z\rangle\cong C_4\times C_2.
\end{align*}
Now Lemma~1.7 in \cite{Olsson} yields $\A_0(D,b_D)=\{M_1,M_2,M_3,D\}$. Assume that $M_1$ and $M_2$ are conjugate in $G$. Then also the $B$-subpairs $(M_1,b_{M_1})$ and $(M_2,b_{M_2})$ are conjugate. By Alperin's fusion theorem they are already conjugate in $\N_G(D,b_D)$. Since $e(B)=1$, this is not possible. 

Now we determine a system of representatives for the conjugacy classes of $B$-subsections using (6C) in \cite{Brauerstruc}. As usual, one gets four major subsections $(1,B)$, $(x^2,b_{x^2})$, $(z,b_z)$, $(x^2z,b_{x^2z})$. Then $b_{x^2}$ dominates a block with defect group $D/\langle x^2\rangle\cong C_2^3$. Since $e(B)=1$, we get $l(b_{x^2})=1$. On the other hand, $b_z$ and $b_{x^2z}$ dominate blocks with defect group $D_8$. 

Since $\Aut(M_3)$ is a $2$-group, we have $\N_G(M_3,b_{M_3})=D\C_G(M_3)$. This gives two subsections $(x,b_x)$ and $(xz,b_{xy})$. Again we have $l(b_x)=l(b_{xz})=1$.

If $\N_G(M_1,b_{M_1})=D\C_G(M_1)$ and $\N_G(M_2,b_{M_2})=D\C_G(M_2)$, then $B$ would be nilpotent. Thus, we may assume $\N_G(M_1,b_{M_1})/\C_G(M_1)\cong S_3$. Then the elements $\{y,x^2y,yz,x^2yz\}$ are conjugate to elements of $\Z(D)$ under $\N_G(M_1,b_{M_1})$. Hence, there are no subsections corresponding to the subpair $(M_1,b_{M_1})$ (cf. Lemma~2.10 in \cite{OlssonRedei}). We distinguish two cases.

\textbf{Case 1:} $\N_G(M_2,b_{M_2})=D\C_G(M_2)$.\\
Then the action of $\N_G(M_2,b_{M_2})$ gives the subsections $(xy,b_{xy})$ and $(xyz,b_{xyz})$. Moreover, $l(b_{xy})=l(b_{xyz})=1$ holds. Since $\N_G(M_1,b_{M_1})$ fixes exactly one element of $\{z,x^2z\}$, we get $l(b_z)+l(b_{x^2z})=3$ (see Theorem~2 in \cite{Brauer}) 
Collecting all the subsections, we deduce $k(B)=l(B)+8$. We may assume that $l(b_z)=2$ (otherwise replace $b_z$ with $b_{x^2z}$). Then the Cartan matrix of $b_z$ is equivalent to $\bigl(\begin{smallmatrix}6&2\\2&6\end{smallmatrix}\bigr)$ (see page~294/5 in \cite{Erdmann}). This gives $k(B)\le 10$. Since $16$ is not the sum of $9$ positive squares, we must have $k(B)=10$. Then $k_0(B)=8$, $k_1(B)=2$ and $l(B)=2$.
In order to determine the Cartan matrix, we investigate the generalized decomposition numbers $d^u_{\chi\phi}$ first. For $u\in D$ with $l(b_u)=1$ we write $\IBr(b_u)=\{\phi_u\}$. Then the numbers $\{d^{x^2}_{\chi\phi_{x^2}}:\chi\in\Irr(B)\}$ can be arranged in the form
\[(1,1,1,1,1,1,1,1,2,2)^{\text{T}},\]
where the last two characters have height $1$. 
It is easy to see that the subsections $(x,b_x)$ and $(x^{-1},b_x)$ are conjugate by $y$. This shows that the numbers $d^x_{\chi\phi_x}$ are integral. The same holds for $d^{xz}_{\chi\phi_{xz}}$. Hence, all irreducible characters are $2$-rational. For every character $\chi$ of height $0$ we have $d^x_{\chi\phi_x}\ne 0\ne d^{xz}_{\chi\phi_{xz}}$. Hence, we get three columns of the generalized decomposition matrix:
\[\begin{pmatrix}
1&1&1&1&1&1&1&1&2&2\\
1&1&1&1&-1&-1&-1&-1&.&.\\
1&1&-1&-1&1&1&-1&-1&.&.
\end{pmatrix}^{\text{T}}.\]
Adding the columns $\{d^{xy}_{\chi\phi_{xy}}:\chi\in\Irr(B)\}$ and $\{d^{xyz}_{\chi\phi_{xyz}}:\chi\in\Irr(B)\}$ gives:
\[\begin{pmatrix}
1&1&1&1&1&1&1&1&2&2\\
1&1&1&1&-1&-1&-1&-1&.&.\\
1&1&-1&-1&1&1&-1&-1&.&.\\
1&-1&1&-1&1&-1&1&-1&.&.\\
1&-1&1&-1&-1&1&-1&1&.&.
\end{pmatrix}^{\text{T}}\]
(To achieve this form, one may has to interchange the third row with fifth and the fourth with the sixth as well as the second column with the third.)
Since $(x^2z,b_{x^2z})$ is a major subsection, the column $\{d^{x^2z}_{\chi\phi_{x^2z}}:\chi\in\Irr(B)\}$ consists of eight entries $\pm1$ and two entries $\pm2$. However, there are three essentially different ways to add this column to the previous ones:
\[\begin{pmatrix}
1&1&1&1&1&1&1&1&2&2\\
1&1&1&1&-1&-1&-1&-1&.&.\\
1&1&-1&-1&1&1&-1&-1&.&.\\
1&-1&1&-1&1&-1&1&-1&.&.\\
1&-1&1&-1&-1&1&-1&1&.&.\\
1&1&1&1&1&1&1&1&-2&-2
\end{pmatrix}^{\text{T}}\]
or
\[\begin{pmatrix}
1&1&1&1&1&1&1&1&2&2\\
1&1&1&1&-1&-1&-1&-1&.&.\\
1&1&-1&-1&1&1&-1&-1&.&.\\
1&-1&1&-1&1&-1&1&-1&.&.\\
1&-1&1&-1&-1&1&-1&1&.&.\\
1&-1&-1&1&1&-1&-1&1&2&-2
\end{pmatrix}^{\text{T}}\]
or
\[\begin{pmatrix}
1&1&1&1&1&1&1&1&2&2\\
1&1&1&1&-1&-1&-1&-1&.&.\\
1&1&-1&-1&1&1&-1&-1&.&.\\
1&-1&1&-1&1&-1&1&-1&.&.\\
1&-1&1&-1&-1&1&-1&1&.&.\\
1&-1&-1&1&-1&1&1&-1&2&-2
\end{pmatrix}^{\text{T}}\]

We use a GAP to enumerate the remaining columns corresponding to the subsection $(z,b_z)$. In all cases the Cartan matrix of $B$ is equivalent to 
\[\begin{pmatrix}6&2\\2&6\end{pmatrix}.\]

\textbf{Case 2:} $\N_G(M_2,b_{M_2})/\C_G(M_2)\cong S_3$.\\
Then one can see by the same argument as for $(M_1,b_{M_1})$ that there are no subsections corresponding to the subpair $(M_2,b_{M_2})$. Since $\N_G(M_1,b_{M_1})$ and $\N_G(M_2,b_{M_2})$ fix exactly one element of $\{z,x^2z\}$ (not necessarily the same), we have $l(b_z)+l(b_{x^2z})=4$ (the cases $l(b_z)=l(b_{x^2z})=2$, $l(b_z)=3$, $l(b_{x^2z})=1$ and $l(b_z)=1$, $l(b_{x^2z})=3$ are possible). We deduce $k(B)=l(B)+7$. If $l(b_z)=2$, then we get $k(B)\le 10$ as in Case~1. Assume $l(b_z)=3$. Then the Cartan matrix of $b_z$ is equivalent to
\[2\begin{pmatrix}2&1&0\\1&3&1\\0&1&2\end{pmatrix}.\]
Thus, also in this case we have $k(B)\le 10$. A consideration of the lower defect groups shows that $2$ occurs as elementary divisor of the Cartan matrix $C$ of $B$. In particular $l(B)\ge 2$ and $k(B)\ge 9$.
Since $16$ is not the sum of $9$ positive squares, it follows that $k(B)=10$, $k_0(B)=8$, $k_1(B)=2$ and $l(B)=3$. 
An investigation of the generalized decomposition numbers similar as in the first case reveals that $C$ is equivalent to
\[\begin{pmatrix}4&2&0\\2&6&2\\0&2&4\end{pmatrix}.\]
This proves the proposition.
\end{proof}

It is easy to see that both cases ($l(B)\in\{2,3\}$) in \autoref{d8xc2} occur for the principal blocks of $S_4\times C_2$ and $\GL(3,2)\times C_2$ respectively.

\begin{Proposition}\label{q8xc2}
Let $B$ be a nonnilpotent block with defect group $Q_8\times C_2$. Then $k(B)=14$, $k_0(B)=8$, $k_1(B)=6$ and $l(B)=3$. The ordinary irreducible characters are $2$-rational. The Cartan matrix of $B$ is equivalent to 
\[\begin{pmatrix}8&4&4\\4&8&4\\4&4&8\end{pmatrix}.\]
In particular the $k(B)$-conjecture holds for defect groups which are central extensions of $Q_8\times C_2$ by a cyclic group.
\end{Proposition}
\begin{proof}
Let $D:=\langle x,y,z\mid x^2=y^2,\ xyx^{-1}=y^{-1},\ z^2=[x,z]=[y,z]=1\rangle\cong Q_8\times C_2$ and let $(D,b_D)$ a Sylow subpair. Since $|\Z(D):\Phi(D)|=2$, we have $e(B)\in\{1,3\}$. As in the proof of \autoref{d8xc2} there are precisely three selfcentralizing proper subgroups of $D$:
\begin{align*}
M_1&:=\langle x,z\rangle\cong C_4\times C_2,\\
M_2&:=\langle y,z\rangle\cong C_4\times C_2,\\
M_3&:=\langle xy,z\rangle\cong C_4\times C_2.
\end{align*}
It follows from Lemma~1.7 in \cite{Olsson} that $\A_0(D,b_D)=\{M_1,M_2,M_3,D\}$. Since $\Aut(M_i)$ is a $2$-group for $i=1,2,3$, $B$ would be nilpotent if $e(B)=1$. Thus, we may assume that $e(B)=3$. Then $M_1$, $M_2$ and $M_3$ are conjugate in $G$. We describe a system of representatives for the conjugacy classes of $B$-subsections. As usual, there are four major subsections $(1,B)$, $(x^2,b_{x^2})$, $(z,b_z)$ and $(x^2z,b_{x^2z})$. Moreover, the subpair $(M,b_M)$ gives the subsections $(x,b_x)$ and $(xz,b_{xz})$. The blocks $b_z$ and $b_{x^2z}$ dominate blocks with defect group $D/\langle z\rangle\cong D/\langle x^2z\rangle\cong Q_8$. Since $\N_G(D,b_D)$ centralizes $\Z(D)$, these blocks with defect group $Q_8$ have inertial index $3$. Now Theorem~3.17 in \cite{Olsson} gives $l(b_z)=l(b_{x^2z})=3$. The block $b_{x^2}$ covers a block with defect group $D/\langle x^2\rangle\cong C_2^3$ and inertial index $3$. Thus, we also have $l(b_{x^2})=3$. Finally the blocks $b_x$ and $b_{xz}$ have defect group $M_1$. Hence, they are nilpotent, and we have $l(b_x)=l(b_{xz})=1$. This yields $k(B)=11+l(B)$. Since $B$ is a centrally controlled block, we get $l(B)\ge l(b_z)=3$ and $k(B)\ge 14$. The Cartan matrix of $b_{x^2}$, $b_{x^2z}$ and $b_{z}$ is equivalent to
\[\begin{pmatrix}8&4&4\\4&8&4\\4&4&8\end{pmatrix}\]
(see page~305 in \cite{Erdmann}). Let $Q\in\mathbb{Z}^{k(B)\times 3}$ be the part of the generalized decomposition matrix corresponding to $b_z$. Then the columns of $Q$ have one of the following forms: $(\pm2,\pm2,0,\ldots,0)$, $(\pm2,\pm1,\pm1,\pm1,\pm1,0,\ldots,0)$ or $(\pm1,\ldots,\pm1,0,\ldots,0)$. Since $k(B)\ge 14$, at least one column has the last form. A similar argument shows that no column has the first form. It follows that at least two columns have the form $(\pm1,\ldots,\pm1,0,\ldots,0)$. Hence, there are four possibilities for $Q$:
\[\begin{array}{cccc}
\left(\begin{array}{ccc}1&.&.\\1&.&.\\1&.&.\\1&.&.\\1&1&2\\1&1&1\\1&1&1\\1&1&.\\.&1&.\\.&1&.\\.&1&.\\.&1&.\\.&.&1\\.&.&1\end{array}\right)& \left(\begin{array}{ccc}1&.&.\\1&.&.\\1&.&.\\1&.&.\\1&1&1\\1&1&1\\1&1&1\\1&1&1\\.&1&1\\.&1&-1\\.&1&.\\.&1&.\\.&.&1\\.&.&1\end{array}\right)& 
\left(\begin{array}{ccc}1&.&.\\1&.&.\\1&.&1\\1&.&1\\1&1&1\\1&1&1\\1&1&.\\1&1&.\\.&1&1\\.&1&1\\.&1&.\\.&1&.\\.&.&1\\.&.&1\end{array}\right)& 
\left(\begin{array}{ccc}1&.&.\\1&.&.\\1&.&.\\1&.&.\\1&1&1\\1&1&1\\1&1&1\\1&1&1\\.&1&.\\.&1&.\\.&1&.\\.&1&.\\.&.&1\\.&.&1\\.&.&1\\.&.&1\end{array}\right)\\
(a)&(b)&(c)&(d)\end{array}\]
In particular $k(B)\in\{14,16\}$ and $l(B)\in\{3,5\}$. 

By way of contradiction, we assume $k(B)=16$. Then $Q$ is given as in case $(d)$. Let $M_z=(m_{\chi\psi}^{(z,b_z)})$ be the matrix of contributions corresponding to $(z,b_z)$. We denote the three irreducible Brauer characters of $b_z$ by $\phi_1,\phi_2$ and $\phi_3$. Then for $\chi\in\Irr(B)$ we have
\begin{align*}16m_{\chi\chi}^{(z,b_z)}&=3\bigl((d_{\chi\phi_1}^{z})^2+(d_{\chi\phi_2}^z)^2+(d_{\chi\phi_3}^z)^2\bigr)-2d_{\chi\phi_1}^zd_{\chi\phi_2}^z-
2d_{\chi\phi_1}^zd_{\chi\phi_3}^z-2d_{\chi\phi_2}^zd_{\chi\phi_3}^z\\
&\equiv d_{\chi\phi_1}^{z}+d_{\chi\phi_2}^z+d_{\chi\phi_3}^z\pmod{2}.\end{align*}
In particular the numbers $16m_{\chi\chi}^{(z,b_z)}$ are odd for all $\chi\in\Irr(B)$. Now (5G) in \cite{BrauerBlSec2} implies $k(B)=k_0(B)$. By Proposition~1 in \cite{BroueSanta} we get $d^x_{\chi\phi_x}\ne 0$ for all $\chi\in\Irr(B)$. However, $\sum_{\chi\in\Irr(B)}{|d_{\chi\phi_x}^x|^2}=|M_1|=8$. 

This contradiction yields $k(B)=14$ and $l(B)=3$. The last argument gives also $k_0(B)\le 8$. Now a similar analysis of the contributions reveals that $Q$ has the form $(c)$ (see above)
and $k_0(B)=8$. Again (5G) in \cite{BrauerBlSec2} implies $k_1(B)=6$ (this follows also from Corollary~1.4 in \cite{Landrock2}). Since the subsections $(x,b_x)$ and $(x^{-1},b_x)$ are conjugate in $G$, the generalized decomposition numbers $d^x_{\chi\phi_x}$ and $d^{xz}_{\chi\phi_{xz}}$ are integral. Thus, they must consist of eight entries $\pm1$ (for the characters of height $0$) and six entries $0$. In particular all characters are $2$-rational. Now we enumerate all possible decomposition matrices with GAP. In all cases the Cartan matrix of $B$ has the stated form.
\end{proof}

The principal block of $\SL(2,3)\times C_2$ gives an example for the last proposition.

\begin{Proposition}
Let $B$ be a nonnilpotent block with defect group $D_8\ast C_4$ (central product). Then $k(B)=14$, $k_0(B)=8$, $k_1(B)=6$ and $l(B)=3$. Moreover, the irreducible characters of height $0$ are $2$-rational and the characters of height $1$ consist of three pairs of $2$-conjugate characters. The Cartan matrix of $B$ is equivalent to 
\[\begin{pmatrix}8&4&4\\4&8&4\\4&4&8\end{pmatrix}.\]
In particular the $k(B)$-conjecture holds for defect groups which are central extensions of $D_8\ast C_4$ by a cyclic group.
\end{Proposition}
\begin{proof}
The proof (and the result) is very similar to \autoref{q8xc2}. Let $D:=\langle x,y,z\mid x^4=y^2=[x,z]=[y,z]=1,\ yxy=x^{-1},\ x^2=z^2\rangle\cong D_8\ast C_4$. We have $e(B)\in\{1,3\}$ and $\A_0(D,b_D)=\{M_1,M_2,M_3,D\}$ with
\begin{align*}
M_1&:=\langle x,z\rangle\cong C_4\times C_2,\\
M_2&:=\langle y,z\rangle\cong C_4\times C_2,\\
M_3&:=\langle xy,z\rangle\cong C_4\times C_2.
\end{align*}
Hence, we may assume $e(B)=3$. Then $M_1$, $M_2$ and $M_3$ are conjugate in $G$. There are four major subsections $(1,B)$, $(z,b_z)$, $(z^{-1},b_{z^{-1}})$ and $(x^2,b_{x^2})$. The subpair $(M_1,b_{M_1})$ gives two nonmajor subsections $(x,b_x)$ and $(xz,b_{xz})$ up to conjugation. As usual, we have $l(b_x)=l(b_{xz})=1$. The blocks $b_z$ and $b_{z^{-1}}$ dominate blocks with defect groups $D/\langle z\rangle\cong C_2^2$ and inertial index $3$. Hence, we have $l(b_z)=l(b_{z^{-1}})=3$. The block $b_{x^2}$ dominates a block with defect group $C_2^3$ and inertial index $3$. Thus, again we have $l(b_{x^2})=3$. Collecting these numbers gives $k(B)=11+l(B)$. The Cartan matrix of the blocks $b_z$, $b_{z^{-1}}$ and $b_{x^2}$ is 
\[\begin{pmatrix}8&4&4\\4&8&4\\4&4&8\end{pmatrix}\]
up to equivalence. Now an analysis of the generalized decomposition numbers $d^{x^2}_{\chi\phi}$ as in the proof of \autoref{q8xc2} reveals $k(B)=14$, $k_0(B)=8$, $k_1(B)=6$ and $l(B)=3$. Next we study the other generalized decomposition numbers. Again as in the proof of \autoref{q8xc2} the numbers $d^{x}_{\chi\phi}$ and $d^{xz}_{\chi\phi}$ are integral. Thus, they consist of eight entries $\pm1$ and six entries $0$. However, in contrast to \autoref{q8xc2} the numbers $d^z_{\chi\phi}$ and $d^{z^{-1}}_{\chi\phi}$ are not always real (see (6B) in \cite{BrauerBlSec2}). Let $Q$ be the part of the generalized decomposition matrix corresponding to $(z,b_z)$. By Brauer's Permutation Lemma, eight of the ordinary irreducible characters are $2$-rational. The remaining ones split in three pairs of $2$-conjugate characters (see Theorem~11 in \cite{Brauerconnection}). This shows that $Q$ has exactly eight real valued rows. 
Let $q_j$ be the $j$-th column of $Q$ for $j=1,2,3$. Then we can write $q_j=a_j+b_ji$ with $i:=\sqrt{-1}$ and $a_j,b_j\in\mathbb{Z}^{14}$. 
The orthogonality relations show that $a_j$ has four entries $\pm1$ and ten entries $0$ (for $j=1,2,3$). 
The same holds for $b_j$. Moreover, we have $4=(q_1\mid q_2)=(a_1\mid a_2)+(b_1\mid b_2)$ and $0=(q_1\mid\overline{q_2})=(a_1\mid a_2)-(b_1\mid b_2)$, where $(.\mid.)$ denotes the standard scalar product of $\mathbb{C}^{14}$. This shows $(a_1\mid a_2)=(b_1\mid b_2)=2$ and similarly $(a_1\mid a_3)=(a_2\mid a_3)=(b_1\mid b_3)=(b_2\mid b_3)=2$. Using this, we see that $Q$ has the form
\[Q=\left(\begin{array}{cccccccccccccc}1&1&1&1&.&.&.&.&i&-i&i&-i&.&.\\1&1&.&.&1&1&.&.&i&-i&.&.&i&-i\\1&1&.&.&.&.&1&1&.&.&i&-i&i&-i\end{array}\right)^{\text{T}}.\]
The theory of contributions reveals that the eight characters of height $0$ are $2$-rational. As in the proof of the previous propositions we enumerate the possible generalized decomposition matrices with GAP, and obtain the Cartan matrix of $B$. 
\end{proof}

We collect the previous propositions in the next theorem.

\begin{Theorem}\label{index16}
Let $B$ be a block with a defect group which is a central extension of a group $Q$ of order $16$ by a cyclic group.
If $Q\not\cong C_2^4$ or $9\nmid e(B)$, then Brauer's $k(B)$-conjecture holds for $B$.
\end{Theorem}
\begin{proof}
For convenience of the reader, we list the $14$ groups of order $16$:
\begin{itemize}
\item the metacyclic groups: $C_{16}$, $C_8\times C_2$, $C_4^2$, $C_4\rtimes C_4$, $D_{16}$, $Q_{16}$, $SD_{16}$ (semidihedral), $M_{16}$ (modular)
\item the minimal nonabelian group: $\langle x,y\mid x^4=y^2=[x,y]^2=[x,x,y]=[y,x,y]=1\rangle$
\item the nonmetacyclic abelian groups: $C_4\times C_2^2$, $C_2^4$
\item $D_8\times C_2$
\item $Q_8\times C_2$
\item $D_8\ast C_4$\qedhere
\end{itemize}
\end{proof}

\begin{Corollary}\label{defect5}
Let $B$ be a block with defect group $D$ of order $32$. If $D$ is not extraspecial of type $D_8\ast D_8$ (central product) or if $9\nmid e(B)$, then Brauer's $k(B)$-conjecture holds for $B$.
\end{Corollary}
\begin{proof}
By \autoref{index16} we may assume that $9\mid e(B)$. In particular $9\mid\Aut(D)$. Now one can show (for example with GAP) that there are just three possibilities for $D$, namely $C_2^5$, $Q_8\times C_2^2$ and the extraspecial group $D_8\ast D_8$. In the case $D\cong Q_8\times C_2^2$ we can choose a major subsection $(u,b)$ such that $D/\langle u\rangle\cong Q_8\times C_2$. 

Hence, by hypothesis we may assume that $D$ is elementary abelian. By Corollary~1.2(ii) in \cite{Robinson} we may also assume that the inertial group $I(B)$ of $B$ is nonabelian. In particular $9$ is a proper divisor of $e(B)$. In general $e(B)$ is a divisor of $3^2\cdot 5\cdot 7\cdot 31$ (this is the odd part of $|\Aut(D)|=|\GL(3,2)|$). 

Assume that $e(B)$ is also divisible by $31$. Since the normalizer of a Sylow $31$-subgroup of $\Aut(D)\cong\GL(5,2)$ has order $5\cdot 31$, $I(B)$ does not contain a normal Sylow $31$-subgroup. Thus, by Sylow's theorem we also have $7\mid e(B)$. However, all groups of order $3^2\cdot 7\cdot 31$ and $3^2\cdot 5\cdot 7\cdot 31$ have a normal Sylow $31$-subgroup. This shows $31\nmid e(B)$.

Now suppose that $5\cdot7\mid e(B)$. Since the normalizer of a Sylow $7$-subgroup of $\GL(5,2)$ has order $2\cdot 3^2\cdot 7$, $I(B)$ does not contain a normal Sylow $7$-subgroup. However, all groups of order $3^2\cdot 5\cdot 7$ have a normal Sylow $7$-subgroup. Hence, $5\cdot 7\nmid e(B)$.

Next we consider the case $e(B)=3^2\cdot 7$. Then the action of $I(B)$ on $D$ induces an orbit of length $21$. If we choose the major subsection $(u,b)$ such that $u$ lies in this orbit, then the inertial index of $b$ is $3$. Thus, the claim follows in this case.

Finally in the case $e(B)=3^2\cdot 5$, the inertial group $I(B)$ would be abelian. Hence, the proof is complete.
\end{proof}

\section{$2$-blocks with minimal nonmetacyclic defect groups}
We remark that the groups $C_2^3$, $Q_8\times C_2$ and $D_8\ast C_4$ are minimal nonmetacyclic. Apart from these there is only one more minimal nonmetacyclic $2$-group (see Theorem~66.1 in \cite{Berkovich2}). Thus, it seems natural to obtain the block invariants also for this defect group. The next proposition shows that these blocks are nilpotent. We use the notion of fusion systems (see \cite{Linckelmann} for definitions and results). 

\begin{Proposition}
Every fusion system on $P:=\langle x,y,z\mid x^4=y^4=[x,y]=1,\ z^2=x^2,\ zxz^{-1}=xy^2,\ zyz^{-1}=x^2y\rangle$ is nilpotent.
\end{Proposition}
\begin{proof}
Let $\mathcal{F}$ be a fusion system on $P$, and let $Q<P$ be an $\mathcal{F}$-essential subgroup. Since $Q$ is metacyclic and $\Aut(Q)$ is not a $2$-group, we have $Q\cong Q_8$ or $Q\cong C_{2^r}^2$ for some $r\in\mathbb{N}$ (see Lemma~1 in \cite{Mazurov}). By Proposition~10.17 and Proposition~1.8 in \cite{Berkovich1} it follows that $Q\cong C_4^2$. Now Theorem~66.1 in \cite{Berkovich2} implies $Q=\langle x,y\rangle$.
As usual, $\Aut_{\mathcal{F}}(Q)\cong S_3$ acts nontrivially on $\Omega_1(Q)$. However, $D$ acts trivially on $\Omega_1(Q)=\Z(D)$. 
This is not possible, since $D/Q$ is a Sylow $2$-subgroup of $\Aut_{\mathcal{F}}(Q)$. Thus, we have shown that $D$ does not contain $\mathcal{F}$-essential subgroups. By Alperin's fusion theorem, $D$ controls $\mathcal{F}$. Finally one can show (maybe with GAP) that $\Aut(D)$ is a $2$-group.
\end{proof}

The group in the last proposition has order $32$.

\begin{Corollary}
 Let $B$ be a $2$-block with minimal nonmetacyclic defect group $D$. Then one of the following holds:
\begin{enumerate}[(i)]
\item $B$ is nilpotent. Then $k_i(B)$ is the number of ordinary characters of $D$ of degree $2^i$. In particular $k(B)$ is the number of conjugacy classes of $D$ and $k_0(B)=|D:D'|$. Moreover, $l(B)=1$.
\item $D\cong C_2^3$. Then $k(B)=k_0(B)=8$ and $l(B)\in\{3,5,7\}$ (all cases occur).
\item $D\cong Q_8\times C_2$ or $D\cong D_8\ast C_4$. Then $k(B)=14$, $k_0(B)=8$, $k_1(B)=6$ and $l(B)=3$.
\end{enumerate}
\end{Corollary}

\end{document}